\documentclass{amsart}
\usepackage{amsmath,latexsym}
\usepackage[psamsfonts]{amssymb}
\usepackage{times}
\usepackage[mathcal]{euscript}
\numberwithin{equation}{section}

\renewcommand{\epsilon}{\varepsilon}

\newcommand{\be}{\begin{equation}}
\newcommand{\ee}{\end{equation}}
\newcommand{\no}{\nonumber}

\newcommand{\spec}{\mathrm{spec}}

\newcommand{\N}{\mathbb{N}}

\newcommand{\Q}{\mathbb{Q}}
\newcommand{\R}{\mathbb{R}}

\newcommand{\T}{\mathbb{T}}

\newcommand{\Z}{\mathbb{Z}}

\newcommand{\cE}{{\mathcal E}}
\newcommand{\cF}{{\mathcal F}}

\newcommand{\cH}{{\mathcal H}}

\renewcommand{\det}{\mathop{\mathrm{det}}}

{\bf}{\it}
\newtheorem{theorem}{Theorem}[section]
\newtheorem{lemma}[theorem]{Lemma}
\newtheorem{corollary}[theorem]{Corollary}

\newtheorem{remark}[theorem]{Remark}

\date{\today}

\begin{document}

\title[Number of bound states of the Schr\"odinger operator of a system of three
bosons...] {Number of bound states of  the Schr\"odinger operator of
a system of three bosons in an optical lattice.}

 \author{Saidakhmat N.\ Lakaev,
 Alimzhan R.\ Khalmukhamedov,\\ Ahmad M.\ Khalkhuzhaev}



\begin{abstract}
We consider the Hamiltonian $\hat {\mathrm{H}}_{\mu}$ of a system of
three identical particles(bosons) on the $d-$ dimensional lattice
$\Z^d, d=1,2$ interacting via pairwise zero-range attractive
potential $\mu<0$. We describe precise location and structure of the
essential spectrum of the Schr\"odinger operator $H_\mu(K),K\in
\T^d$ associated to $\hat {\mathrm{H}}_\mu$ and prove the finiteness
of the number of bound states of $H_\mu(K),K\in \T^d$ lying below
the bottom of the essential spectrum. Moreover, we show that bound
states decay exponentially at infinity and eigenvalues and
corresponding bound states of $H_\mu(K),K\in \T^d$ are regular as a
function of center of mass quasi-momentum $K\in \T^d$.
\end{abstract}

\maketitle


\section{ Introduction}
Cold atoms loaded in an optical lattice provide a realization of a
quantum lattice gas. The periodicity of the potential gives rise to
a band structure for the dynamics of the atoms.

The dynamics of the ultracold atoms loaded in the lower band is well
described by the Bose-Hubbard hamiltonian \cite{Letters}; we  give
in section 2 the corresponding Sch\"odinger operator.

In the continuous case \cite{Faddeev},\cite{FadMer,RSIII} the energy
of the center-of-mass motion can by separated out from the total
Hamiltonian,  i.e., the energy operator can by split into a sum of a
center-of-mass motion and a relative kinetic energy. So that the
three-particle {\it bound states} are eigenvectors of the relative
kinetic energy operator.

The kinematics of the quantum particles on the lattice is rather
exotic. The discrete laplacian \it is not translationally invariant
\rm and therefore one cannot separate the motion of the center of
mass.

One can rather resort to a Bloch-Floquet decomposition. The
three-particle Hilbert space $ \mathcal{H} \equiv \ell ^2 ({\Z}^d)^3
$ is decomposed as direct  integral associated to the representation
of the discrete group $ {\Z}^d $ by shift operators.
\begin{equation*}
\ell ^2[({\Z}^d)^3] = \int_{{\T}^d} \oplus \ell
^2[({\Z}^d)^2]\eta(dK),
\end{equation*}
where $\eta(dp)$ is the (normalized) Haar measure on the torus
$\T^d$. Hence the total three-body Hamiltonian $\mathrm{H}$ of a
system of three particles on $d-$ dimensional lattice $\Z^d, d
\geq1$ interacting via pairwise short range attractive potential $V$
appears to be decomposable
\begin{equation*}\label{decompose}
\mathrm{H}=\int\limits_{\T^d}\oplus H(K)\eta(dK).
\end{equation*}

The fiber hamiltonians $H(K) $ depends parametrically on the {\it
quasi momentum} $ K \in \T^d \equiv \R^d / (2 \pi {\Z}^d) .$ It is
the sum of a free part depending on $K$ continuously and an
interaction term, both bounded.

Bound states $\psi_{E,K}$ are solution of the Schr\"odinger equation
\begin{equation*}
H(K) \psi_{E,K} = E \psi_{E,K}  \qquad \psi_{E,K} \in \ell ^2
[(\Z^d)^2].
\end{equation*}

The Efimov effect is one of the remarkable results in the spectral
analysis of Hamiltonians associated to a system of three-particles
moving on the three-dimensional Euclid space:
if none of the three two-particle Schr\"{o}dinger operators
(associated to the two-particle subsystems of a three-particle
system) has negative eigenvalues, but at least two of them have zero
energy resonance, then the three-particle Schr\"{o}dinger operator
has an infinite number of discrete eigenvalues, accumulating at zero
\cite{ALM,ALzM04,ALKh,Efimov,Lak93,Sob,Tam,Yaf}.


The finiteness of eigenvalues (absence Efimov's effect) have been
proved for the Hamiltonian of a system of three particles moving on
$d=1,2$ dimensional Euclid space $\R^d$ in \cite{VZh83}.

We consider the Hamiltonian $\hat{\mathrm{H}}_{\mu}$ of a system of
three identical particles(bosons) on $d-$ dimensional lattice $\Z^d,
d=1,2,$ interacting via pairwise zero range attractive potential
$\mu<0$.

We prove the finiteness  of the number of bound states(absence of
Efimov's effect) of the Schr\"odinger operator $H_{\mu}(K),\,K\in
\mathbb{G}\subseteq\T^d,d=1,2,$ where $\mathbb{G}\subseteq\T^d$ is a
region, associated to the Hamiltonian $\hat {\mathrm{H}}_\mu$.

We describe a precise location and structure of the essential
spectrum of the Schr\"odinger operator $H_\mu(K),K\in \T^d$.
Moreover, we show that bound states decay exponentially at infinity
and we establish that the eigenvalues and corresponding bound states
of $H_\mu(K),K\in \T^d$ are regular as a function of center of mass
quasi-momentum $K\in \T^d$.

In \cite{MumAli12} finiteness of the eigenvalues of the discrete
Schr\"odinger operator associated to a system of three-bosons on one
dimensional lattice $\Z^1$ has been shown.

Section 1 is an introduction. In Section 2 we introduce the
Hamiltonians of systems of two and three-particles in coordinate and
momentum representations as bounded self-adjoint operators in the
corresponding Hilbert spaces.

In Section 3 we introduce the total quasi-momentum, decompose the
energy operators into von Neumann direct integrals, introduce
discrete Schr\"odinger operators $h_{\mu}(k),k\in \T^d$ and
$H_{\mu}(K),K\in \T^d$, choosing relative coordinate system.

We state the main results in Section 4.

We introduce the {\it channel operators} and describe the essential
spectrum of $H_{\mu}(K)$ by means of the discrete spectrum of the
two particle Schr\"odinger operators $h_\mu(k),k\in \T^d$ (Theorem
\ref{essTheorem}) in section 5.

In Section 6 we prove the finiteness of the number of eigenvalues of
the three-particle Schr\"{o}dinger operator $H_{\mu}(K),K\in \T^d$
(Theorem \ref{finite}) and finiteness of isolated bands in a system
of three particles in an optical lattice.

\section{Hamiltonians of three identical
particles on a lattices in the coordinate and momentum
representations} Let ${\Z}^{d},d=1,2$ be the $d$-dimensional
lattice. Let $\ell^2[({\Z}^{d})^{m}],d=1,2$ be Hilbert space of
square-summable functions $\,\,\hat{\varphi}$ defined on the
Cartesian power of $({\Z}^{d})^{m},d=1,2$ and let
$\ell^{2,s}[(\Z^{d})^{m}]\subset \ell^{2}[(\Z^{d})^{m}]$ be the
subspace of symmetric functions.

Let $\Delta$ be the lattice Laplacian, i.e., the operator which
describes the transport of a particle from one site to another site:
$$ (\Delta\hat{\psi})(x)= \frac{1}{2}\sum_{\mid s\mid =1} [\hat{\psi}(x)-
\hat{\psi}(x+s)],\quad \hat{\psi}\in\ell^2({\Z}^d).$$ The free
Hamiltonian $\hat{\mathrm{h}}_0$ of a system of two identical
quantum mechanical particles with mass $m=1$ on the $d$-dimensional
lattice $\Z^d,d=1,2$ in the coordinate representation is associated
to the self-adjoint operator $\hat{\mathrm{h}}_0$ in the Hilbert
space $\,\,\ell^{2,s}[(\Z^{d})^{2}]$:
\begin{equation}\no
\hat{\mathrm{h}}_0={\Delta}\otimes I + I\otimes\Delta,
\end{equation}
where $\otimes$ denotes the tensor product and $I$ is the identity
operator on $L^2(\Z^d)$.
 The Hamiltonian $\hat {\mathrm{h}}_\mu$ of a system of two
 identical particles  with the two-particle pair zero-range attractive interaction
$\mu\hat v$ is a bounded perturbation of the free Hamiltonian
$\hat{\mathrm{h}}_0$ on the Hilbert space $\ell^{2,s}[( {\Z}^d)^2]$
\begin{equation*}\label{two-part}
\hat{\mathrm{h}}_\mu =\hat{\mathrm{h}}_0+\mu\hat v.
\end{equation*}
Here $\mu<0$ is coupling constant and
\begin{equation*}
(\hat v\hat \psi)(x_1,x_2) = \delta _{x_1 x_2}
{\hat\psi}(x_1,x_2),\quad {\hat\psi} \in \ell^{2,s}[({\Z}^d)^2],
\end{equation*}
where $\delta _{x_1 x_2}$ is  the Kronecker delta.

Similarly, the free Hamiltonian $\hat{\mathrm{H}}_0$ of a system of
three identical particles on the $d$-dimensional lattice $\Z^d$ with
mass $m=1$ is defined on the Hilbert space $\ell^{2,s}[({\Z}^d)^3]$:
\begin{equation*}\label{free0}
\widehat H_0=\Delta\otimes I\otimes I + I \otimes \Delta \otimes I +
I\otimes I\otimes \Delta.
\end{equation*}
The Hamiltonian $ \hat{\mathrm{H}}_\mu $  of a system of three
identical particles with the two-particle pair zero-range
interactions $\hat v=\hat v_{\alpha}=\hat
v_{\beta\gamma},\alpha,\beta,\gamma=1,2,3$ is a bounded perturbation
of the free Hamiltonian $\hat {\mathrm{H}}_0$
\begin{equation*}\label{total}
 \hat {\mathrm{H}}_\mu=\hat {\mathrm{H}}_0+\mu\hat {\mathbb{V}},
\end{equation*}
where $\mathbb{V}=\sum_{\alpha=1}^3 \hat
{V}_{\alpha},\,{V}_{\alpha}=\hat V, \alpha=1,2,3 $ is the
multiplication operator on $\ell^{2,s}[({\Z}^d)^3]$ defined by
\begin{align*}
&(\hat V_{\alpha}\hat\psi)(x_1,x_2,x_3)=  \delta _{x_\beta
x_\gamma}\hat\psi(x_1,x_2,x_3),\\
&\alpha\prec\beta\prec\gamma\prec\alpha,\,\alpha,\beta,\gamma=1,2,3,\,\,
\hat\psi\in \ell^{2,s}[({\Z}^d)^3].\\
\end{align*}
\subsection{ The momentum  representation}
 Let ${\T}^d=(-\pi,\pi]^d$ be the $d-$ dimensional torus and
$L^{2,s}[(\mathbb{T}^{d})^{m}]\subset L^2[(\mathbb{T}^{d})^{m}]$ be
the subspace of symmetric functions defined on the  Cartesian power
${({\T}^d)^m},\,m\in \N$.

Let $\cF: \ L^2(\T^d) \to \ell^2(\Z^d)$  be the standard Fourier
transform
\begin{equation*} \label{eq-0.9}
\cF: \ \ell^2(\Z^d) \to L^2(\T^d), \quad [\cF(f)](p) \ := \ \sum_{x
\in \Z^d} e^{-\mathrm{i} (p, x)} \:  f(x)
\end{equation*}
with the inverse
\begin{equation*} \label{eq-0.10}
\cF^*: \ L^2(\T^d) \to \ell^2(\Z^d), \quad [\cF^*(\psi)](x) \ := \
\int_{ \T^d} e^{\mathrm{i}(p, x)} \: \psi(p) \, \eta(dp)
\end{equation*}
and $\eta(dp) =\frac {d^dp}{(2\pi)^d}$ is the (normalized) Haar
measure on the torus.

It easily can be checked that Fourier transform  \begin{equation*}
 \hat \Delta={\cF}\Delta \ {\cF}^{*}
 \end{equation*} of the Laplacian
$\Delta$ is the multiplication operator by the function
$\varepsilon(\cdot)$:
\begin{equation*}
(\hat \Delta f)(p)= {\varepsilon}(p)f(p),\quad f \in L^2(\T^d),
\end{equation*}
where
\begin{equation*}
\varepsilon(p)=\sum_{i=1}^{d}(1-\cos p^{(i)}),\quad p=(p^{(1)},...,
p^{(d)})\in \T^d.
\end{equation*}
The two-particle total Hamiltonian  $\mathrm{h}_\mu$ in the momentum
representation is given on $L^{2,s}[({\T}^d)^2]$ as follows
\begin{equation*}
\mathrm{h}_\mu = \mathrm{h}_0+ \mu v.
 \end{equation*}
Here the free Hamiltonian $\mathrm{h}_0$ is of the form
\begin{equation*}
\mathrm{h}_0 =\hat \Delta \otimes \hat I+ \hat I \otimes \hat
\Delta,
\end{equation*}
where $\hat I$ is the identity operator on $L^2(\T^d)$. It is easy
to see that the operator $\mathrm{h}_0$ is the multiplication
operator by the function $ \varepsilon (k_1)+\varepsilon (k_2):$
\begin{equation*}
(\mathrm{h}_0f)(k_1 ,k_2)=[\varepsilon(k_1)+\varepsilon (k_2 )]f(k_1
,k_2),\,\,f\in L^{2,s}[({\T}^d)^2]
\end{equation*}
and $v$ is the convolution type integral operator
\begin{align*}
&( vf)(k_1,k_2 )= {\int\limits_{({\T}^d)^2} }
 \delta (k_1
+k_2 -k_1'-k_2')f(k_1',k_2')\eta(dk_1')\eta(dk_2')\\
&=\int\limits_{\T^d} f(k_1',k_1 +k_2-k_1')\eta(dk_1'),\, f\in
L^{2,s}[({\T}^d)^2],
\end{align*}
where $ \delta (\cdot)$ is the $d-$ dimensional Dirac delta
function.

The three-particle Hamiltonian in the momentum representation is
given as bounded self-adjoint operator on the Hilbert space
 $L^{2,s}[({\T}^d)^3]$
\begin{equation*}
\mathrm{H}_\mu= \mathrm{H}_0+\mu({V}_{1}+V_{2}+V_{3}),
\end{equation*}
where  $\mathrm{H}_0$ is of the form
\begin{equation*}
\mathrm{H}_0=\hat \Delta\otimes \hat I\otimes\hat I+\hat I \otimes
 \hat \Delta \otimes \hat I+\hat I\otimes \hat I\otimes  \hat
 \Delta,
\end{equation*}
i.e., the free Hamiltonian $\mathrm{H}_0$ is the multiplication
operator by the function $\sum_{\alpha =1}^d \varepsilon
(k_\alpha)$:
\begin{equation*}
(\mathrm{H}_0f)(k_1,k_2,k_3)  = [\sum_{\alpha =1}^3 \varepsilon
(k_\alpha)]f(k_1,k_2,k_3),
\end{equation*}
and $V_{\alpha}=V,\alpha=1,2$  are convolution type integral
operators
\begin{align*}
&( V_{\alpha}f)(k_{\alpha},k_{\beta},k_{\gamma})\\
&={\int\limits_{({\T}^d)^3} } \delta (k_{\alpha} -k_{\alpha}')\,
\delta (k_{\beta} +k_{\gamma}
-k_{\alpha}'-k_{\beta}')f(k'_\alpha,k'_\beta,k'_\gamma)\eta(dk_{\alpha}')\eta(dk_{\beta}')\eta(dk_{\gamma}')\\
&={\int\limits_{{\T}^d}}f(k_{\alpha},k_{\beta}',k_{\beta}
+k_{\gamma}-k_{\beta}')\eta(dk_{\beta}'), \quad f\in
L^{2,s}[({\T}^d)^3].
\end{align*}

\section{Decomposition of the energy operators into von Neumann direct integrals.
Quasi-momentum and coordinate systems}

Let  $k=k_1+k_2\in \T^d$ resp. $K=k_1+k_2+k_3\in \T^d$ be the {\it
two-} resp. {\it three-particle quasi-momentum} and the set
$\mathbb{Q}_{k}$ resp. $\mathbb{Q}_{K}$ is defined as follows
\begin{align*}
&\mathbb{Q}_{k}=\{(k_1,k-k_1){\in }({\T}^d)^2: k_1 \in\T^d
k-k_1\in\T^d\}
\end{align*}
resp.
\begin{align*}
&\mathbb{Q}_{K}=\{(k_1,k_2,K-k_1-k_2){\in }({\T}^d)^3: k_1,k_2
\in\T^d, K-k_1-k_2\in\T^d\}.
\end{align*}
We introduce the mapping
\begin{equation*}
\pi_{1}:(\T^d)^2\to \T^d,\quad \pi_1(k_1, k_2)=k_1
\end{equation*}
resp.
\begin{equation*}
\pi_2:(\T^d)^3\to (\T^d)^2,\quad \pi_2(k_1, k_2, k_3)=(k_1, k_2).
\end{equation*}

Denote by  $\pi_k$, $k\in \T^d$ resp. $\pi_{K}$ , $K\in \T^d$ the
restriction of $\pi_1$ resp. $\pi_2$ onto $\Q_{k}\subset (\T^d)^2,$
resp. $\Q_{K}\subset (\T^d)^3$, that is,
\begin{equation*}\label{project}\pi_{k}= \pi_1\vert_{\Q_{k}}\quad
\text{and}\quad \pi_{K}=\pi_{2}\vert_{\Q_{K}}.
\end{equation*}
It is useful to remark that $ \Q_{k},\,\, k \in {\T}^d $ resp.
 $ \Q_{K},\,\, K \in
{\T}^d $ are the $d-$ resp. $2d-$ dimensional manifold isomorphic to
${\T}^d$ resp.
 ${({\T}^d)^2}$.

\begin{lemma}
The mapping   $\pi_{k}$, $k\in \T^d$ resp. $\pi_{K}$, $K\in \T^d$ is
bijective from $\mathbb{Q}_{k}\subset (\T^d)^2$ resp.
$\mathbb{Q}_{K}\subset (\T^d)^3$ onto  $\T^d$ resp. $(\T^d)^2$ with
the inverse mapping given by
\begin{equation*}
(\pi_{k})^{-1}(k_1)=(k_1,k-k_1)
\end{equation*}
resp.
\begin{equation*}
(\pi_{K})^{-1}(k_1,k_2)= (k_1, k_2, K-k_1-k_2).
\end{equation*}
\end{lemma}
Let $L^{2,e}({\T}^d)\subset L^2 ( {\T}^d)$ be the subspace of even
functions. Decomposing the Hilbert space $ L^{2,s}[(\T^d)^2]$ resp.
$ L^{2,s}[({\T}^d)^3]$  into the direct integral
\begin{equation*}\label{tensor}
L^{2,s}[({\T}^d)^2] = \int_{k\in {\T}^d} \oplus
L^{2,e}({\T}^d)\eta(dk)
\end{equation*}
resp.
\begin{equation*}
L^{2,s}[({\T}^d)^3] = \int_{K\in {\T}^d} \oplus
L^{2,s}[({\T}^d)^2]\eta(dK)
\end{equation*}
yields the  decomposition of the Hamiltonian $\mathrm{h}_\mu$ resp.
$\mathrm{H}_\mu$ into the direct integral
\begin{equation*}\label{fiber2}
\mathrm{h}_\mu= \int\limits_{k \in \T^d}\oplus\tilde
h_\mu(k)\eta(dk)
\end{equation*}
resp.
\begin{equation*}\label{fiber3}
\mathrm{H}_\mu=\int\limits_ {K \in {\T}^d}\oplus \tilde
H_\mu(K)\eta(dK).
 \end{equation*}
\subsection{The discrete Schr\"odinger operators} The fiber operator $\tilde h_\mu(k),$
$k \in {\T}^d$ is unitarily equivalent to the operators $h_\mu(k),$
$k \in {\T}^d$ acting in $L^{2,e}({\T}^d)\subset L_2 ({\T}^d)$:
\begin{equation}\label{two} h_\mu(k) =h_0(k)+\mu v.
\end{equation}
The operator $h_0(k)$ is the multiplication operator by the function
$\cE_k(p)$:
\begin{equation*}(h_0(k)f)(p)=\cE_k(p)f(p),\quad f \in L^{2,e}(\T^d),
\end{equation*}
where
\begin{equation*}\label{E-alpha}
\cE_{k}(p)= \varepsilon (\frac{k}{2}-p) +\varepsilon
 (\frac{k}{2}+p)
\end{equation*}
and
\begin{equation*}
(vf)(p)= \int\limits_{{\T}^d}f(q)d \eta (q), \quad f \in
L^{2,e}({\T}^d).
\end{equation*}

The fiber operator $\tilde H_\mu(K),\ K \in {\T}^d$ is unitarily
equivalent to the operator $H_\mu (K),$ $K \in {\T}^d$ given by
\begin{equation*}\label{three-particle}
H_\mu(K)=H_0(K)+\mu (V_1+V_2+ V_3),V_\alpha=V,\ \alpha=1,2,3.
\end{equation*}
The operator $H_0(K),\ K \in {\T}^d$ acts in the Hilbert space
$L^{2,s}[({\T}^d)^2]$ and is of the form
\begin{equation*}\label{TotalK}
(H_0(K)f)(p,q)=E(K;p,q)f(p,q),\quad f\in  L^{2,s}[({\T}^d )^2],
\end{equation*}
where
\begin{equation*}\label{multyp_funct} E(K;p,q)= \varepsilon
(K-p-q)+\varepsilon(p) + \varepsilon (q).
\end{equation*}
The operator $\mathbb{V}=V_1+V_2+V_3$ acts in $L^{2,s}[({\T}^d)^2]$
and in the coordinates $(p,q)\in (\T^d)^2$ can be written as follows
\begin{align*}\label{potential}
(\mathbb{V}f)(p,q)&=\int\limits_{\T^d}f(p,t)
\eta(dt)+\int\limits_{\T^d}f(t,q)\eta(dt)+\int\limits_{\T^d}f(t,K-p-q)\eta(dt),f\in
L^{2,s}
[({\T}^d)^2].\\
\end{align*}
\section{Statement of the main results}

According to Weyl's theorem \cite{RSIV} the essential  spectrum
$\sigma_{\mathrm{essspec}}(h_\mu(k))$ of the operator $h_\mu(k),\,k
\in \T^d$ coincides with the spectrum $ {\sigma}( h_0 (k) ) $ of
$h_0(k).$ More precisely,
$$
\sigma_{\mathrm{essspec}}(h_\mu(k))= [\cE_{\min}(k)
,\,\cE_{\max}(k)],
$$
where \begin{align*} &\cE_{\min}(k)\equiv\min_{p\in
\T^d}\cE_k(p)=2\sum_{i=1}^{d}[1-\cos(\frac{k^{(i)}}{2})]\\
&\cE_{\max}(k)\equiv\max_{p\in\T^d}
\cE_{k}(p)=2\sum_{i=1}^{d}[1+\cos(\frac{k^{(i)}}{2})].
\end{align*}

The following Theorem states the existence of a unique eigenvalue of
the operator $h_\mu(k)$.

\begin{theorem}\label{existencetwo} For any $\mu<0$ and $k \in \T^d,\,d=1,2$\, the operator $h_\mu(k)$ has a
unique eigenvalue $e_\mu(k)$, which is even on $\T^d$ and satisfies
the relations $e_\mu(k)<\cE_{\min}(k),\,k\in\T^d$ and
$e_\mu(0)<e_\mu(k), k \ne0$. Moreover, for any $\mu<0$ the
eigenvalue $e_\mu(k)$ is regular function in $k\in\T^d$.

The associated eigenfunction $f_{\mu,e_\mu(k)}(\cdot),k \in \T^d$ is
a regular function on ${\T}^d$ and has the form
\begin{equation*}\label{eigen}
f_{\mu,e_\mu(k)}(\cdot)=\frac{\mu c(k)}{\cE_{k}(\cdot)-e_\mu(k)},
\end{equation*}
where $c(k)\neq 0,\,k \in \T^d$ is  a normalizing constant.
Moreover, the vector valued mapping
\begin{equation*}\label{map}
f_{\mu}:\mathbb{\T}^d \rightarrow
L^2[\mathbb{\T}^d,\eta(dk);L^{2,e}({\T}^d)],\,k\rightarrow
f_{\mu,e_\mu(k)}
\end{equation*} is regular on $\mathbb{\T}^d$.
\end{theorem}
Theorem \ref{existencetwo} can be proven in the same way as
in\cite{LXL}.

We note that the spectrum $\sigma_{\mathrm{spec}}(H_0(K))$ of the
operator $H_0(K),\ K \in \T^d$ is the segment $[\mathrm{E}_{\min
}(K),\mathrm{E}_{\max}(K)]$, where
\begin{align*}
\mathrm{E}_{\min }(K)\equiv\min_{p,q\in \T^d}E(K;p,q),\
\mathrm{E}_{\max }(K)\equiv\max_{p,q \in \T^d}E(K;p,q).
\end{align*}

We describe the essential spectrum of the three-particle operator
$H_\mu(K),\,K \in \T^d$ by the spectrum of the non perturbed
operator $H_0(K),\ K \in \T^d$ and the discrete spectrum of the
two-particle operator $h_\mu(k),\,k \in \T^d$ in the following
theorem.

\begin{theorem}\label{essTheorem} For
any $\mu<0$ the essential spectrum $\sigma_{\mathrm{essspec}}(H_\mu
(K))$ of $H_\mu(K),\ K \in \T^d$ is described  as follows
$$
\sigma_{\mathrm{essspec}}(H_\mu (K))=\cup _{k\in
{\T}^d}\{e_\mu(k)+\varepsilon (K-k)\} \cup
[E_{\min}(K),E_{\max}(K)],
$$
where $e_\mu(k)$ is a unique eigenvalue of the operator
$h_\mu(k),k\in \T^d$.
\end{theorem}
The next theorem states the finiteness of the number of eigenvalues
for the Schr\"odinger operator $H_\mu(K)$ and the analyticity of the
eigenvalues and associated eigenfunctions.

Let $U_{\delta(K)}[p_\mu (K)]=\{K \in \T^d:|K-p_\mu (K)|<\delta\}$
be $\delta=\delta(K)-$ neighborhood of the point $p_\mu (K)\in
\T^d$.

\begin{theorem}\label{finite} Let $d=1,2$ and $\mu<0.$ Then
\begin{enumerate}
\item[(i)] There exists $\delta>0$ such that for each $K \in U_{\delta}[0]$
the operator $H_\mu(K)$ has  finite number of eigenvalues
$E_{1,\mu}(K),...,E_{n,\mu}(K)$ lying below the bottom of the
essential spectrum $\sigma_{\mathrm{essspec}}(H_{\mu}(K))$ with the
associated bound states
$$\psi_{\mu,E_{1,\mu}(K)}(\cdot),...,\psi_{\mu,E_{n,\mu}(K)}(\cdot)\in
L^{2,s}[({\T}^d)^2],\,K\in U_{\delta}[0].$$
\item[(ii)]The eigenfunction
$f_{\mu,E_\mu(K)}(\cdot,\cdot)\in L^{2,s}[({\T}^d)^2]$ of $H_\mu(K)$
associated to the eigenvalue $E_\mu(K), K\in U_{\delta}[0]$ is
regular in $(p,q)\in(\T^d)^2$.  Moreover, each eigenvalue
$E_\mu(K),K\in U_{\delta}[0]$ of $H_\mu(K)$ is a regular function in
$K\in U_{\delta}[0]$ and the vector valued mapping
\begin{equation*}\label{map}
f_{\mu}:U_{\delta}[0] \rightarrow
L^2[U_{\delta}[0],\eta(dK);L^{2,s}({\T}^d)^2],\,K\rightarrow
f_{\mu,E_\mu(K)}
\end{equation*} is also regular on $\mathbb{\T}^d$.
\end{enumerate}
\end{theorem}
\begin{corollary} The two-particle Hamiltonian $\mathrm{h}_\mu$ has
a unique isolated band spectrum and the three-particle Hamiltonian
$\mathrm{H}_\mu$ have a finite number band spectrum.
\end{corollary}
Denote by $\tau_{\spec}(H_{\mu}(K))$ resp.
$\tau_{\mathrm{essspec}}(H_{\mu}(K))$ the bottom of the spectrum
resp. essential spectrum of the three-particle Schr\"odinger
operator $H_\mu(K),\ K\in \T^d$,\, i.e.,
\begin{equation*}\label{bottomSpec}
\tau_{\mathrm{spec}}(H_{\mu}(K))= \inf_{||f||=1}(H_\mu(K)f,f).
\end{equation*}
resp.
\begin{equation}\label{bottomEssspec}
\tau_{\mathrm{essspec}}(H_{\mu}(K))=\inf\sigma_{\mathrm{essspec}}(H_\mu
(K)).
\end{equation}
Let
\begin{equation*}\label{esstwo}
\sigma_{\mathrm{esstwo}}(H_\mu (K))=\cup _{k\in
{\T}^d}\{e_\mu(k)+\varepsilon (K-k)\}.
\end{equation*}
resp.
\begin{equation*}\label{essthree}
\sigma_{\mathrm{essthree}}(H_{\mu}(K))= [E_{\min}(K),E_{\max}(K)]
\end{equation*}
 be the two-particle resp. three-particle
essential spectrum of $H_\mu (K),\ K\in \T^d$ and
\begin{equation*}\label{bottomess}
\tau_{\mathrm{esstwo}}(H_{\mu}(K))=\inf\sigma_{\mathrm{esstwo}}(H_\mu
(K))
\end{equation*}
resp.
\begin{equation*}\label{botthreepart}
\tau_{\mathrm{essthree}}(H_{\mu}(K))=
E_{\min}(K)=\inf_{||f||=1}[(H_0(K)f,f)]
\end{equation*}
be the bottom of the two-particle resp.\ three-particle essential
spectrum.
\begin{remark}
For the operator $H_{\mu}(K)$ associated to a system of three bosons
on the lattice $\Z^d,d=1,2$ Theorems \ref{existencetwo} and
\ref{essTheorem} give
\begin{equation*}\label{d=1ord=2}
\sigma_{\mathrm{esstwo}}(H_\mu (K))\neq\emptyset
\end{equation*}
and
\begin{equation*}
\tau_{\mathrm{esstwo}}(H_{\mu}(K))<
\tau_{\mathrm{essthree}}(H_{\mu}(K))
\end{equation*}
and hence
\begin{equation*}
\sigma_{\mathrm{essthree}}(H_{\mu}(K))\subset\sigma_{\mathrm{essspec}}(H_{\mu}(K)).
\end{equation*}
Consequently, the inequality
\begin{equation*}\label{d=1or d= 2}
\tau_{\mathrm{essspec}}(H_{\mu}(K))<
\tau_{\mathrm{essthree}}(H_{\mu}(K))
\end{equation*}
holds, which allows to prove the finiteness of the number of bound
states of three interacting bosons on the lattice $\Z^d,d=1,2$.
\end{remark}

\begin{remark} We remark that for the three-particle Schr\"odinger
operator $H_{\mu}(K)$, associated to a system of three bosons in the
three-dimensional lattice $\Z^3$, there exists $\mu_0>0$ such that
\begin{equation*}\label{Efimov}
\sigma_{\mathrm{essspec}}(H_{\mu_0}(0))=\sigma_{\mathrm{essthree}}(H_{\mu_0}(0))
\end{equation*}
and hence
\begin{equation*}\label{d=1or d= 2}
\tau_{\mathrm{essspec}}(H_{\mu_0}(0))=
\tau_{\mathrm{essthree}}(H_{\mu_0}(0)).
\end{equation*}
At the same time for any nonzero $K\in\T^3$ the following relation
\begin{equation*}\label{d=1ord=2}
\sigma_{\mathrm{esstwo}}(H_{\mu_0}(K))\neq\emptyset
\end{equation*}
holds and hence
\begin{equation*}
\tau_{\mathrm{esstwo}}(H_{\mu_0}(K))<
\tau_{\mathrm{essthree}}(H_{\mu_0}(K))
\end{equation*}
and
\begin{equation*}
\sigma_{\mathrm{essthree}}(H_{\mu_0}(K))\subset\sigma_{\mathrm{essspec}}(H_{\mu_0}(K)).
\end{equation*}
Thus only the operator $H_{\mu_0}(0)$ may have an infinite number of
eigenvalues below the bottom of the three-particle continuum
(Efimov's effect)\cite{ALzM04,Lak93}, which yields the existence of
an infinite number of bound states.
\end{remark}

\section{The essential spectrum of the operator $H_\mu(K)$.}
Since we are considering the system of identical particles, there is
only one channel operator
 $ H_{\mu,ch}(K), K{\ \in } \T^d,\,d=1,2$ defined in the
Hilbert space $L^{2,s}[({\T}^d )^2]= L^2({\T}^d)\otimes L^{2,e}
({\T}^d)$ as
\begin{equation*}
 H_{\mu,ch}(K)=H_0(K)+\mu V.
\end{equation*}
The operators $H_0(K)$ and $V=V_\alpha$ act as follows
\begin{equation*}\label{TotalK}
(H_0(K)f)(p,q)=\cE (K;p,q)f(p,q),\quad f\in L^{2,s}[({\T}^d )^2],
\end{equation*}
where
\begin{equation*}\label{Eps}
\cE(K;p,q)= \varepsilon (K-p)+\varepsilon  (\frac{p}{2}-q) +
\varepsilon (\frac{p}{2}+q)
\end{equation*}
and
 \begin{equation*}\label{Poten}(V
f)(p,q)= \int\limits_{\T^d}f(p,t)\eta(dt),\quad f\in L^{2,s}[({\T}^d
)^2].
\end{equation*}

The decomposition of the space $L^{2,s}[(\T^d)^2]$ into the direct
integral

 $$L^{2,s}[({\T}^d )^2]= \int\limits_{k\in \T^d}
\oplus L^{2,e}( \T^d)\eta(dk)
$$
yields for the operator $H_{\mu,ch}(K)$ the decomposition
 $$H_{\mu,ch}(K)=\int\limits_{k\in \T^d}
 \oplus h_{\mu}(K,k) \eta(dk).$$
The fiber operator $h_{\mu}(K,k)$ acts in the Hilbert space
$L^{2,e}(\T^d)$ and has the form
\begin{equation}\label{representation}
 h_{\mu}(K,k) =h_{\mu}
(k)+\varepsilon(K-k) I,
\end{equation} where $I=I_{L^{2,e}(\T^d)}$ is
the identity operator and $h_{\mu}(k)$ is the two-particle operator
defined by \eqref{two}. The representation \eqref{representation} of
the operator $h_{\mu}(K,k)$ and Theorem \ref{existencetwo} yield for
the spectrum of operator $h_{\mu}(K,k)$ the equality
\begin{align}\label{stucture}
 &\sigma (h_{\mu}(K,k))
 = Z_\mu(K,k)\cup \big[\cE_{\text{min}}(k),\cE_{\text{max}}(k) \big],
\end{align}
 where
\begin{equation*}\label{ZK}
Z_\mu(K,k)=e_\mu(k)+\varepsilon(K-k)
\end{equation*} and $e_\mu(k)$ is
the unique eigenvalue of the operator $h_{\mu}(k)$.

The spectrum of the channel operator $H_{\mu,ch}(K),\ K\in \T^d$ is
described in the following
\begin{lemma}\label{relation} The
equality holds
\begin{equation*}\label{structur}
\sigma (H_{\mu,ch}(K))=\cup _{k\in \T^d} \left \{Z_\mu(K,k) \right\}
\cup [E_{\min}(K),E_{\max}(K)].
\end{equation*}
\end{lemma}
\begin{proof}
The theorem (see, e.g.,\cite{RSIV}) on the spectrum of decomposable
operator and the structure \eqref{stucture} of the spectrum of
$h_{\mu}(K,k)$ give the proof.
\end{proof}
The essential spectrum of $H_\mu (K),\ K\in \T^d$ is described in
the following
\begin{theorem}\label{equality}The equality
$$\sigma _{\mathrm{essspec}}(H_\mu(K))=\sigma (H_{\mu,ch}(K))$$ holds.
\end{theorem}
\begin{proof} Theorem \ref{equality} can be proven by the same way as
Theorem 3.2
 in \cite{ALzM04}.
\end{proof}
Theorems \ref{existencetwo} and \ref{equality} yield that the bottom
$\tau_{\mathrm{essspec}}(H_{\mu}(K))$ of the essential spectrum of
the operator $H_\mu (K)$ less than the bottom
$\tau_{\mathrm{spec}}(H_0(K))=\tau_{\mathrm{essthree}}(H_{\mu}(K))$
of the spectrum of the non-perturbed operator $H_0(K)$, which is
attribute for the three-particle Schr\"odinger operators on the
lattice $\Z^d$ and Euclid space $\R^d$ in dimensions $d=1,2$.

\begin{lemma}\label{inequality}  For any $\mu<0$ and $K\in \T^d,d=1,2$
the bottom of the essential spectrum of $H_{\mu}(K)$ satisfies the
relations
$$\tau_{\mathrm{essspec}}(H_{\mu}(K))<
\tau_{\mathrm{essthree}}(H_{\mu}(K))=\tau_{\mathrm{spec}}(H_0(K))=E_{\min}
(K)$$ holds, where $\tau_{\mathrm{essspec}}(H_{\mu}(K))$ is defined
in \eqref{bottomEssspec}.
\end{lemma}
\begin{proof} Theorem \ref{existencetwo} yields that for any
$k\in \T^d$ the operator $h_\mu(k)$ has a unique eigenvalue
$e_\mu(k)<2\, \varepsilon(\frac{k}{2})=\cE_{\min}(k).$

Hence
$$
Z_{\mu}(K,k)\Big|_{k=\frac{2K}{3}}= e_{\mu}(\frac{2K}{3})+
\varepsilon(\frac{K}{3})<
2\varepsilon(\frac{K}{3})+\varepsilon(\frac{K}{3})=
3\varepsilon(\frac{K}{3})= E_{\min}(K).
$$
The definition of $\tau_{\mathrm{essspec}}(H_{\mu}(K))$ gives
\begin{align*}
&\tau_{\mathrm{essspec}}(H_{\mu}(K))=
\tau_{\mathrm{spec}}(H^{ch}_{\mu}(K))\\
&=\inf\sigma(H^{ch}_\mu(K))=\inf_{k\in\T^d} Z_{\mu}(K,k) \leq
e_{\mu}(\frac{2K}{3})+
\varepsilon(\frac{K}{3})<3\varepsilon(\frac{K}{3}),
\end{align*}
which proves Lemma \ref{inequality}.
\end{proof}
\section{Proof of the main results}

For any $K,k\in \T^d,d=1,2$ the essential  spectrum
   $\sigma_{\mathrm{essspec}}(h_\mu(K,k))$ of the
operator $h_\mu(K,k),\,K,k \in \T^d$ coincides with the spectrum $
{\sigma}( h_0 (K,k) ) $ of $h_0(K,k).$ More precisely,
$$
\sigma_{\mathrm{essspec}}(h_\mu(K,k))= [\cE_{\min}(K,k)
,\,\cE_{\max}(K,k)].
$$
\begin{align*}
&E_{\min}(K,k)=\min_{q}E(K,k\,;q)=\min_{q}\cE_k(q)+\varepsilon(K-k)=2\varepsilon(\frac{k}{2})+\varepsilon(K-k)\\
&E_{\max}(K,k)=\max_{q}E(K,k\,;q)=\max_{q}\cE_k(q)+\varepsilon(K-k)=[2d-\varepsilon(\frac{k}{2})]+\varepsilon(K-k).
\end{align*}
The determinant $\Delta_\mu (K,k\,;z),\ K,k\in \T^d,d=1,2$
associated to the operator $h_{\mu}(K,k)$ can be defined as an
regular function in $\mathrm{C}\setminus [E_{\min }(K,k),\,E_{\max
}(K,k)]$ as
\begin{equation*}\label{determinant}
\Delta_\mu (K,k\,;z)= 1+\mu \int\limits_{\T^d}\frac{\eta(dk)}{E
(K,k\,;q)-z}.
\end{equation*}
Let $\mathrm{L}_\mu( K,z),\,\, K\in \T^d,\,\, z <
\tau_{\mathrm{essspec}}(H_{\mu}(K))$ be a self-adjoint operator
defined in  $L^{2}(\T^d)$ as
\begin{equation*}\label{compact_operator}
[\mathrm{L}_\mu(K,z)w](p)=-\mu
 \int\limits_{\T^d} \frac{\Delta^{-\frac{1}{2}}_{\mu}(K,p,z)
\Delta_{\mu}^{-\frac{1}{2}}(K,q,z)}{E(K;p,q)-z}w(q)
\eta(dq),\\
w\in L_ 2(\T^d).
\end{equation*}
The operator $\mathrm{L}_\mu(K,z)$ is a lattice analogue of the
Birman-Schwinger operator that has been introduced in \cite{Lak93}
to investigate Efimov's effect for the three-particle lattice
Schr\"{o}dinger operator $H_\mu(K)$.

For a bounded self--adjoint operator $A$ in a Hilbert space
${{\mathcal{H}}}$
and for each $\gamma \in \mathbb{R}$ we define the number $n_{+}[\gamma, A]$ resp. $%
n_{-}[\gamma, A]$ as
\begin{align*}
&n_{+}[\gamma ,A]:\\
&=\max \{\dim{\mathcal{H}}^+_A(\gamma):{\mathcal
{H}}^+_A(\gamma)\subset {{\mathcal{H}}}\text{ subspace with
}\left\langle A\varphi ,\varphi \right\rangle >\gamma ,\text{
}\varphi \in {\mathcal{H}}^+_A(\gamma),\text{ }||\varphi ||=1\}
\end{align*}%
resp.
\begin{align*}
&n_{-}[\gamma ,A]:\\
&=\max \{\dim{\mathcal{H}}^-_A(\gamma):{\mathcal
{H}}^-_A(\gamma)\subset {{\mathcal{H}}}\text{ subspace with
}\left\langle A\varphi ,\varphi \right\rangle <\gamma ,\text{
}\varphi \in {\mathcal {H}}^-_A(\gamma),\text{ }||\varphi ||=1\}.
\end{align*}%
If some point of the essential spectrum of \ $A$ is greater resp.
smaller than $\gamma$, then $n_{+}[\gamma, A]$ resp.
$n_{-}[\gamma,A]$ is equal to infinity. If $n_{+}[\gamma,A]$ resp.
$n_{-}[\gamma,A]$ is finite, then it is equal to the number of the
eigenvalues (counting multiplicities) of $A$, which are greater
resp. smaller than $\gamma $ (see, for instance, Glazman lemma
\cite{Pankov A}).
\begin{remark}
Theorem \ref{essTheorem} yields that for any $K\in \T^d$ the
operator $H_\mu(K)$ has no essential spectrum below
$\tau_{\mathrm{essspec}}(H_{\mu}(K))$.
\end{remark}

\begin{lemma}\label{Bir-Sch}(The
Birman-Schwinger principle). For each $\mu<0,\ K\in \T^d$ and
$z<\tau_{\mathrm{essspec}}(H_{\mu}(K))$ the operator
$\mathrm{L}_\mu(K,z)$ is compact and the equality
$$
n_-[z,H_\mu(K)]=n_+[1,\mathrm{L}_\mu(K,z)].
$$ holds. Moreover for any $\mu<0,\ K\in \T^d$ the operator $\mathrm{L}_\mu(K,z)$ is
continuous in $z\in (-\infty,\tau_{\mathrm{essspec}}(H_{\mu}(K)))$.
\end{lemma}
\begin{proof} We first verify the equality
\begin{equation}\label{tenglik}
n_-[z,H_\mu(K)]=n_+[1,-3\mu R^{\frac{1}{2}}_0(K,z)
VR^{\frac{1}{2}}_0(K,z)].
\end{equation}
 Assume that $u \in
{\cH}^-_{H_\mu(K)}(z)\subset L^{2,s}[(\T^d)^2]$, that is,
$((H_0(K)-z)u,u) < -3\mu( Vu,u).$ Then
$$(y,y) <(-3\mu R^{\frac{1}{2}}_0(K,z)V R^{\frac{1}{2}}_0(K,z)y,y),
\quad y=R^{\frac{1}{2}}_0(K,z)u,
$$
where $R_0(K,z)$ is the resolvent of the $H_0(K)$.Hence
$$n_-[z,H_\mu(K)] \leq
n_+[1,-3 \mu R^{\frac{1}{2}}_0(K,z)V R^{\frac{1}{2}}_0(K,z)].$$
Reversing the argument we get the opposite inequality, which proves
\eqref{tenglik}.

Note that any nonzero  eigenvalue of
$R^{\frac{1}{2}}_0(K,z)V^{\frac{1}{2}}$ is an eigenvalue for
$V^{\frac{1}{2}}R^{\frac{1}{2}}_0(K,z)$ as well, of the same
algebraic and geometric multiplicities. Therefore  we get
$$
n_+[1,-3\mu R^{\frac{1}{2}}_0(K,z) V R^{\frac{1}{2}}_0(K,z)]=
n_+[1,-3\mu V^{\frac{1}{2}} R_0(K,z)V^{\frac{1}{2}}].
$$
Let us check the equality
$$
n_+[1,-3\mu V^{\frac{1}{2}}
R_0(K,z)V^{\frac{1}{2}}]=n_+[1,\mathrm{L}_\mu(K,z)].
$$
We show that for any $$u \in {\cH}_{-3\mu V^{\frac{1}{2}}
R_0(K,z)V^{\frac{1}{2}}}^+(1)$$ there exists
$y\in{\cH}_{\mathrm{L}_\mu(K,z)}^+(1)$ such that
$(y,y)<(\mathrm{L}_\mu (K,z) y,y).$

Let $u\in {\cH}_{-3\mu
V^{\frac{1}{2}}R_0(K,z)V^{\frac{1}{2}}}^+(1).$ Then
$$
(u,u)< -3\mu(V^{\frac{1}{2}} R_0(K,z)V^{\frac{1}{2}} u,u)
$$
and
\begin{equation}\label{coordinate}
([I+\mu V^{\frac{1}{2}} R_0(K,z) V^{\frac{1}{2}
}]u,u)<-2\mu(V^{\frac{1}{2} }R_0(K,z)V^{\frac{1}{2}}u,u).
\end{equation}

Since $z<\tau_{\mathrm{essspec}}(H_{\mu}(K))$ the operator $I+\mu
V^{\frac{1}{2}}R_0(K,z)V^{\frac{1}{2}}$ is invertible and positive
the operator $W^{\frac{1}{2}}_\mu(K,z)=(I+\mu
V^{\frac{1}{2}}R_0(K,z)V ^{\frac{1}{2}})^{-\frac{1}{2}}$ exists.
Setting $$y=(I+\mu V^{\frac{1}{2}}R_0(K,z)V
^{\frac{1}{2}})^{\frac{1}{2}}u $$ gives us
$$
(y,y)< -2\mu(W^{\frac{1}{2}}_\mu(K,z)V^{\frac{1}{2}} R_0(K,z)
V^{\frac{1}{2}} W^{\frac{1}{2}}_\mu(K,z)y,y).
$$
Since $W^{\frac{1}{2}}_\mu(K,z)$ is the multiplication operator by
the function $\Delta^{-\frac12}_\mu(K,p\:;z)$ the inequalities $$
(y,y)\leq (\mathrm{L}_\mu(K,z)y,y) $$ and  $$ n_+[1,-3\mu
R^{\frac{1}{2}}_0(K,z)V R^{\frac{1}{2}}_0(K,z)] \leq
n_+(1,\mathrm{L}_\mu(K,z))$$ hold. By the same way one  checks that
$$ n_+(1,\mathrm{L}_\mu(K,z)) \leq n_+(1,-3\mu
R^{\frac{1}{2}}_0(K,z)V R^{\frac{1}{2}}_0(K,z)).$$
\end{proof}
The following lemma gives the well known relation between the
eigenvalues of $h_{\mu}(K,k)$ and zeros of the determinant
$\Delta_\mu (K,k\,;z)$ \cite{ALzM04}.
\begin{lemma}\label{nollar2}
For all   $K,k\in\T^d$ the number  $z \in  {\mathrm{C}} {\setminus}
[E_{\min }(K,k),\,E_{\max}(K,k)]$ is an eigenvalue of the operator
$h_{\mu}(K,k) $ if and only if
$$
 \Delta_\mu (K, k\,; z) = 0.
$$
 \end{lemma}
The proof of Lemma \ref{nollar2} is usual and can be found
\cite{LXL}.

\begin{lemma}\label{main}
The following assertions ~(i)--(iv) hold true.
\begin{enumerate}
\item[(i)] If $f\in L^{2,s}[({\T}^d)^2]$ solves $H_
\mu(K)f = zf$, $z<\tau_{\mathrm{essspec}}(H_{\mu}(K))$ then
$$\psi(p)=\Delta^{\frac12}_\mu(K,p\:;z)\varphi(p),\, \mbox{where}
\,\,\varphi(p)=\int\limits_{\T^d}f(p,t)\eta(dt)\in L^2({\T}^d)$$
solves $\psi = \mathrm{L_\mu}(k,z)\psi$.

\item[(ii)] If $\psi \in L^2({\T}^d)$ solves $\psi =
\mathrm{L_\mu}(k,z)\psi$, then
\begin{equation*}
f(p,q)=\dfrac{\mu[\varphi(p)+\varphi(q)+\varphi(K-p-q)]}{E(K;p,q)-z}\in
L^{2,s}[({\T}^d)^2],
\end{equation*}
where $\varphi(p)=\Delta^{-\frac12}_\mu(K,p\:;z)\psi(p)$, solves the
equation $H_\mu(K)f = zf$.
\item[(iii)] For any $\mu<0$ the eigenvalue $E_\mu(K)<\tau_{\mathrm{essspec}}(H_{\mu}(K))$ of the operator
$H_\mu(K)$ and the associated eigenfunction $f\in
L^{2,s}[({\T}^d)^2]$ are regular in $K\in \T^d$.
\end{enumerate}
\end{lemma}
\begin{proof}
\begin{enumerate}
\item[(i)] Let for some $K\in\T^d$ and $z<\tau_{\mathrm{essspec}}(H_{\mu}(K))$
the equation
\begin{equation}(H_{\mu}(K)\psi)(p,q)=z\psi(p,q),
\end{equation} i.e., the equation
\begin{align}\label{t1}
&[E(K;p,q)-z]f(p,q)\\ \nonumber
&=-\mu[\;\int\limits_{\T^d}f(p,t)\eta(dt)+\int\limits_{\T^d}f(t,q)\eta(dt)+\int\limits_{\T^d}f(K-p-q,t)\eta(dt)]
\end{align}
has a solution $f\in L^{2,s}[({\T}^d)^2]$.

Denoting by
$$\varphi(p)=\int\limits_{\T^d}f(p,t)\eta(dt)$$ we
rewrite the equation \eqref{t1} as follows
\begin{equation}\label{solution}
f(p,q)=-\mu\frac{\varphi(p)+\varphi(q)+\varphi(K-p-q)}{E(K;p,q)-z}\,\in
L^{2,s}[({\T}^d)^2],
\end{equation}
which gives for $\varphi \in L^2({\T}^d)$ the equation
\begin{equation}\label{b-s}
\varphi(p)=-\mu\int\limits_{\T^d}
\dfrac{\varphi(p)+\varphi(t)+\varphi(K-p-t)}{E(K;p,t)-z}\eta(dt).
\end{equation}
Since the function $E(K;p,t)$ is invariant under  $K-p-t\rightarrow
t$, we have
\begin{align}
\varphi(p)\left[1+\mu\int\limits_{\T^d}
\dfrac{dq}{E(K;p,q)-z}\right]=2
\mu\int\limits_{\T^d}\dfrac{\varphi(q)}{E(K;p,q)-z}\eta(dq)\notag
\end{align}
Denoting by $\Delta^{\frac12}_\mu(K,p\:;z)\varphi(p)=\psi(p)$ and
taking into account the inequality
$\Delta_\mu(K,p\:;z)\neq0,\,z<\tau_{\mathrm{essspec}}(H_{\mu}(K))$
we get the equation
\begin{equation}\label{B-S}
\psi(p)=-2\mu\int\limits_{\T^d}
\dfrac{\Delta^{-\frac12}_{\mu}(K,p\:;z)\Delta^{-\frac12}_{\mu}(K,q\:;z)\psi(q)}{E(K;p,q)-z}\eta(dq).
\end{equation}
\item[(ii)] Assume that for some $z<\tau_{\mathrm{essspec}}(H_{\mu}(K))$
the function  $\psi\in L^2(\T^d)$ is a solution of the equation
\eqref{B-S}.Then $\varphi(p)=\Delta^{-\frac12}(K,p\:;z)\psi(p)\in
L^2(\T^d)$ is a solution of the equation \eqref{b-s}. Hence the
function defined by \eqref{solution} belongs $L^{2,s}[({\T}^d)^2]$
and is a solution of the Schr\"odinger equation $H_\mu(K)f = zf$,
i.e., $f$ is an eigenfunction of the operator $H_{\mu}(K)$
associated to the eigenvalue
$z<\tau_{\mathrm{essspec}}(H_{\mu}(K)).$

\item[(iii)]For all $\mu<0,\,K\in\T^d$ and
$z<\tau_{\mathrm{essspec}}(H_{\mu}(K))$ the kernel function
$$\mathrm{L}_{\mu}(K,z;p,q)=-2\mu \frac{\Delta^{-\frac{1}{2}}_\mu(K,p,z)
\Delta^{-\frac{1}{2}}_\mu(K,q, z)}{E(K;p,q)-z}$$ of the operator
$\mathrm{L}_{\mu}(K,z)$ is continuous in $p,q \in \T^d$. Hence, for
any $\mu<0$ and $K\in\T^d$ the Fredholm determinant
$D_\mu(K,z)=\det[I-\mathrm{L}_{\mu}(K,z)]$ associated to
$\mathrm{L}_{\mu}(K,z;p,q)$ is real and regular function in $z\in
(-\infty,\tau_{\mathrm{essspec}}(H_{\mu}(K)))$.

Lemma \ref{main} and the Fredholm theorem yield that each eigenvalue
$E_\mu(K)\in(-\infty,\tau_{\mathrm{essspec}}(H_{\mu}(K)))$ of the
operator $H_{\mu}(K)$ is a zero of the determinant $D_\mu(K,z)$ and
vice versa. Consequently, the compactness of the torus $\T^d$ and
implicit function theorem yield that  for each $\mu<0$ the
eigenvalue $E_\mu(K)$ of $H_{\mu}(K)$ is a regular function in $K
\in \T^d,\,d=1,2$.

Since the functions $\Delta_\mu(K,p\:;E_\mu(K))$ and
$E(K;p,q)-E_\mu(K)$ are regular in $K\in \T^d$ the solution $\psi\in
L^2[\mathbb{T}^d]$ of the equation \eqref{B-S} and hence the
function $\varphi$ are regular in $K\in \T^d$. Hence, the
eigenfunction \eqref{solution} of the operator $H_{\mu}(K)$
associated to eigenvalue
$E_\mu(K)<\tau_{\mathrm{essspec}}(H_{\mu}(K))$ is also regular in
$K\in \T^d$. Consequently, the vector valued mapping
\begin{equation*}\label{map}
f_{\mu}:\mathbb{T}^d \rightarrow
L^2[\mathbb{T}^d,\eta(dK);L^{2,s}[({\T}^d)^2]],\,K\rightarrow
f_{\mu,K}(\cdot,\cdot)
\end{equation*} is regular in $\mathbb{T}^d$.
\end{enumerate}
\end{proof}

Now we are going to proof the finiteness of the number
$N(K,\tau_{\mathrm{essspec}}(H_{\mu}(K)))$ of eigenvalues of the
three-particle Schr\"odinger operator $H_\mu(K),K\in U_{\delta}[0]$.

We postpone the proof of the main theorem after the following two
lemmas.
\begin{lemma}\label{estimate} Let $d=1,2.$ For any $K \in U_\delta(0)$  there are
positive nonzero constants $C_1$ and $C_2$ depending on $K$ and a
neighborhood $U_{\delta(K)}[p_\mu (K)]$ of the point $p_\mu (K)\in
\T^d$ such that for all $p\in U_{\delta(K)}[p_\mu (K)]$ the
following inequalities
\begin{equation*}\label{otsenka2}
 C_1|p -p_\mu (K)|^2 \leq
\Delta_\mu (K,p,\tau_{\mathrm{essspec}}(H_{\mu}(K)))\leq
C_2|p-p_\mu(K)|^2
\end{equation*}
hold.
\end{lemma}
\begin{proof} We prove Lemma \ref{otsenka2} for the case $d=2$.
The point $p=0$ is the non degenerate minimum of the function
$\varepsilon(p)$, i.e.,
\begin{equation*} \label{eq1}
\varepsilon(p)=\frac{1}{2}p^2+O(|p|^4)\,\,\mbox{as}\,\,p\to 0.
\end{equation*}
Since the eigenvalue $e_\mu(p)$ lying below the essential spectrum
is a unique zero of the determinant $\Delta_\mu(p,z)$ associated to
operator $h_\mu(p)$,  simple computations gives
$$
\bigg ( \frac{\partial^2 e_\mu(0)}{\partial p^{(i)} \partial
p^{(j)}}\bigg )_{i,j=1}^2= C\bigg (\begin{array}{lll}
1\,\,0\\
0\,\,1\\
\end{array}
\bigg ),\,C>0.
$$
Analogously, the eigenvalue $Z_\mu(0,p)$ of the operator
$h_\mu(0,p)$, lying below the essential spectrum, is unique zero of
the determinant $\Delta_\mu(0,p,z)$  and hence the point
$p=p_\mu(0)=0\in \T^2$ is non-degenerate minimum of
\begin{equation*}
Z_\mu(0,p):= e_\mu(p)+\varepsilon(p).
\end{equation*}
Therefore for any $K \in U_\delta(0)$ the point $p_\mu(K)\in
U_\delta(0)$ is non degenerate minimum of the function $Z_\mu(K,p)$
and the matrix
$$
B(K)=\bigg(\frac{\partial^2 Z_{\mu}(K,p_{\mu}(K))}{\partial p^{(i)}
\partial p^{(j)}}\bigg )_{i,j=1}^2
$$
is positive definite. Hence, the eigenvalue $Z_\mu(K,p)$ has
following asymptotics
\begin{equation}\label{Z}
Z_\mu(K,p)=\tau_{\mathrm{essspec}}(H_{\mu}(K))+
(B(K)(p-p_\mu(K)),p-p_\mu(K)) +o(|p-p_\mu(K)|^2),
\end{equation}
as $|p-p_\mu(K)| \to 0,$ where
$\tau_{\mathrm{essspec}}(H_{\mu}(K))=Z_\mu(K,p_\mu(K)).$

For any $K,p \in \T^d$ there exists a $\gamma=\gamma(K,p)>0$
neighborhood $W_\gamma(Z_\mu(K,p))$ of the point $Z_\mu(K,p)\in
\mathrm{C}$ such that for all $z \in W_\gamma(Z_\mu(K,p))$ the
following equality holds
\begin{equation*}
\Delta_\mu (K,p,z)=\sum_{n=1}^{\infty}C_n(\mu,
K,p)(z-Z_\mu(K,p))^n,\\
\end{equation*}
where
\begin{align*}
C_1(\mu,K,p)=-\mu \int\limits_{\T^d}
\dfrac{dq}{[E(K;p,q)-Z_\mu(K,p)]^2}< 0.
\end{align*}
From here one can concludes that for any $K\in U_\delta(0)$ there is
$U_{\delta(K)}(p_{\mu}(K))$ and for all $p\in
U_{\delta(K)}(p_{\mu}(K))$ the equality
\begin{equation}\label{nondeger}
\Delta_\mu(K,p,\tau_{\mathrm{essspec}}(H_{\mu}(K)))
=[Z_\mu(K,p)-\tau_{\mathrm{essspec}}(H_{\mu}(K))]\hat
\Delta_\mu(K,p,\tau_{\mathrm{essspec}}(H_{\mu}(K)))
 \end{equation}
holds, where $\hat
\Delta_\mu(K,p_\mu(K),\tau_{\mathrm{essspec}}(H_{\mu}(K))\neq 0.$
Putting \eqref{nondeger} into \eqref{Z} proves Lemma \ref{estimate}.
\end{proof}
\begin{lemma}\label{compact} Let $K\in U_\delta(0)$.
The operator $\mathrm{L}_\mu(K,z)$ can be represented  as sum of the
two operators
$$\mathrm{L}_\mu(K,z)=\mathrm{L}_\mu^{(1)}(K, z) +\mathrm{L}_\mu^{(2)}(K,z),$$ where
the operator $\mathrm{L}_{\mu}^{(1)}(K,z),
z<\tau_{\mathrm{essspec}}(H_{\mu}(K)$ has finite rank and
$\mathrm{L}_\mu^{(2)} ( K,z),
z\leq\tau_{\mathrm{essspec}}(H_{\mu}(K)$ belongs to the
Hilbert-Schmidt class.\end{lemma}
\begin{proof}
We represent the operator $\mathrm{L}_\mu (K,z)$ as sum of two
operators
$$\mathrm{L}_\mu(K, z)=\mathrm{L}_\mu^{(1)}(K,z)
+\mathrm{L}_\mu^{(2)}(K,z),$$ where
\begin{align*}
&[\mathrm{L}_\mu^{(1)}( K, z)w](p)=2\mu\int\limits_{\T^d}
\frac{\mathrm{L}_\mu^{(1)}(K,z;p,q)w(q) \eta (dq)}
{\Delta^{\frac{1}{2}}_\mu(K,p, z) \Delta^{\frac{1}{2}}_\mu(K, q,
z)},
\end{align*}
is the finite rank operator, where
\begin{align*} &\mathrm{L}_\mu^{(1)}(K,z;p,q)\\
&=\frac{1}{E(K;p,p_\mu (K))-z}+\frac{1}{E(K;p_\mu(K),q)-z}
-\frac{1}{E(K;p_\mu (K),p_\mu (K))-z}
\end{align*}
and
\begin{align*}\label{comp}
&[\mathrm{L}_\mu^{(2)}( K, z)w](p)= 2\mu
\int\limits_{\T^d}\frac{\mathrm{L}_\mu^{(2)}(K,z;p,q)w(q)
 \eta(dq)}{\Delta^{\frac{1}{2}}_\mu(K, p, z \Delta^{\frac{1}{2}}_\mu (K, q,
z)},
\end{align*}
where
\begin{equation*}
\mathrm{L}_\mu^{(2)}( K, z; p,q)=
\frac{1}{E(K;p,q)-z}-\mathrm{L}_\mu^{(1)}(K,z;p,q).
\end{equation*}
For any $z<\tau_{\mathrm{essspec}}(H_{\mu}(K))$ the kernel
$\mathrm{L}_\mu^{(2)}(K,z;p,q)$ of the operator
$\mathrm{L}_\mu^{(2)}(K, z)$ is a regular function at the point
$(p_\mu (K),p_\mu (K))$ and $\mathrm{L}_\mu^{(2)}( K, z; p_\mu
(K),p_\mu (K))=0$.

Lemma \ref{inequality} yields the inequality
\begin{equation*}\label{inequal3} E(K;p,q)
-\tau_{\mathrm{essspec}}(H_{\mu}(K))\geq
E_{\min}(K)-\tau_{\mathrm{essspec}}(H_{\mu}(K))>0.
\end{equation*}
Hence the functions
$$[E(K;p,q)-\tau_{\mathrm{essspec}}(H_{\mu}(K))]^{-1}-\mathrm{L}_\mu^{(1)}(K,\tau_{\mathrm{essspec}}(H_{\mu}(K));p,q)$$
is regular in $(p,q) \in \T^2.$

So, the operator
$\mathrm{L}_\mu^{(2)}(K,\tau_{\mathrm{essspec}}(H_{\mu}(K))), K\in
U_\delta(0)$ belongs to the Hilbert-Schmidt class.
\end{proof}

{\bf Proof of Theorem \ref{finite}.} For any compact operators
$A_1,\, A_2$ and  positive numbers  $\lambda_1,\,\lambda_2$ Weyl's
inequality
$$n_+[\lambda_1+\lambda_2,A_1+A_2]\leq
n_+[\lambda_1,A_1]+n_+[\lambda_2,A_2]$$ gives that for all
$z<\tau_{\mathrm{essspec}}(H_{\mu}(K)),\ K\in U_\delta(0)$ the
relations
$$n_+[1,\mathrm{L}_\mu(K,z)]\leq n_+[\frac{1}{3},
\mathrm{L}_\mu^{(1)}(K, z)]+ n_+[\frac{2}{3},\mathrm{L}_\mu^{(2)}(
K,z])\leq
$$
$$\leq n_+[\frac{1}{3}
,\mathrm{L}_\mu^{(1)}(K,z)]+n_+[\frac{1}{3},\mathrm{L}_\mu^{(2)}(
K,\tau_{\mathrm{essspec}}(H_{\mu}(K)))]+$$$$+
n_+[\frac{1}{3},\mathrm{L}_\mu^{(2)}(K,z)-\mathrm{L}_\mu^{(2)}(
K,\tau_{\mathrm{essspec}}(H_{\mu}(K))].
$$ hold.
The compactness of $\mathrm{L}_\mu^{(2)}(
K,\tau_{\mathrm{essspec}}(H_{\mu}(K))), K\in U_\delta(0)$ yields
$$n_+[\frac{1}{3},\mathrm{L}_\mu^{(2)}(K,\tau_{\mathrm{essspec}}(H_{\mu}(K)))]<\infty.$$
The operator $\mathrm{L}_\mu^{(2)}(K,z)$ is continuous in $z\leq
\tau_{\mathrm{essspec}}(H_{\mu}(K))$ and converges to \\
$\mathrm{L}_\mu^{(2)}(K,\tau_{\mathrm{essspec}}(H_{\mu}(K)))$ in
uniform operator topology. Therefore for all sufficiently small
$|\tau_{\mathrm{essspec}}(H_{\mu}(K))-z|$ we have
$$n_+[\frac{1}{3},\mathrm{L}_\mu^{(2)}(K,z)-\mathrm{L}_\mu^{(2)}(
K,\tau_{\mathrm{essspec}}(H_{\mu}(K)))]=0.$$

The dimension of rank of $\mathrm{L}_\mu^{(1)}(K,z), z<
\tau_{\mathrm{essspec}}(H_{\mu}(K))$ does not depend on $z<
\tau_{\mathrm{essspec}}(H_{\mu}(K))$ and
$$n_+[\frac{1}{3} ,\mathrm{L}_\mu^{(1)}(K,z)]<\infty.$$
Lemma \ref{Bir-Sch} yields that the operator
$\mathrm{L}_\mu(K,\tau_{\mathrm{essspec}}(H_{\mu}(K))), K\in
U_\delta(0)$ has a finite number eigenvalues greater than $1$ and
consequently by Lemma \ref{Bir-Sch} the operator $H_\mu(K),K\in
U_\delta(0)$ has a finite number of eigenvalues in the interval
$(-\infty, \tau_{\mathrm{essspec}}(H_{\mu}(K))).$

From Lemma \ref{compact} one concludes that if
$E_\mu(K)<\tau_{\mathrm{essspec}}(H_{\mu}(K))$ is an eigenvalue of
the operator $H_{\mu}(K)$ then the associated eigenfunction is of
the form \eqref{solution}.

{\bf Acknowledgements} This work was supported by the Fundamental
Science Foundation of Uzbekistan.The first author would like to
thank Fulbright program for  supporting the research project during
2013-2014 academic year.


\begin{thebibliography}{99}

\bibitem{ALM}  {\sc S.~Albeverio, S.N.~Lakaev, K.~A.~Makarov}: The Efimov Effect and
an Extended Szego-Kac Limit Theorem, Letters in Math.Phys,V.{\bf
43}(1998), 73-85.
\bibitem{ALMM06}  {\sc S.~ Albeverio, S.~N.~Lakaev, K.~A.~Makarov,
 Z.~I.~Muminov}: { The Threshold Effects for the Two-particle Hamiltonians on Lattices,}
 Comm.Math.Phys. {\bf 262}(2006), 91--115 .

\bibitem{ALzM04}  {\sc S.~ Albeverio, S.~N.~Lakaev and
 Z.~I.~Muminov}: { Schr\"{o}dinger operators on lattices. The Efimov effect and
discrete spectrum asymptotics.} Ann. Henri Poincar\'{e}. {\bf 5},
(2004),743--772.
\bibitem{ALKh} {\sc Albeverio, S. N. Lakaev and A. M. Khalkhujaev:} Number of eigenvalues
of the three-particle Schr\"odinger operators on lattices, Markov
Process. Related Fields, 18(3):387--420 2012 ISSN: 1024-2953

\bibitem{Efimov} {\sc V. Efimov :} Energy Levels Arising from Resonant Two-Body Forces in
a Three-Body System, Physics Letters, {\bf 33}B(8), 1970 pp.
563–564.

\bibitem{Faddeev}  {\sc L. D. Faddeev}: Mathematical aspects of the three--body
problem in quantum mechanics. Israel Program for Scientific
Translations, Jerusalem, 1965.
\bibitem{FadMer}  {\sc L. D. Faddeev} and {\sc S. P. Merkuriev}: Quantum
scattering theory for several particle systems. Kluwer Academic
Publishers, 1993.

states of the cluster operator, Theor.and Math.Phys.{\bf 39}(1979),
No.1,336-342.
\bibitem{Lak93}  {\sc S.~N.~Lakaev}: {The Efimov's Effect of a system of Three
Identical Quantum lattice Particles,} Funkcionalnii analiz i ego
priloj. , {\bf 27}(1993), No.3, pp.15-28, translation in Funct.
Anal.Appl.
\bibitem{LXL}  {\sc ~S.~N.Lakaev, A.M.Khalkhuzhaev and Sh.S.Lakaev:} Asymptotic behavior of an eigenvalue
of the two-particle discrete Schr\"{o}dinger operator. Theor. Math.
Phys. {\bf 171(3)}(2012), 800-811.
\bibitem{Mat86}  {\sc D.C.Mattis}:The few-body problem on lattice, Rev.Modern
Phys. {\bf 58}(1986), No. 2, 361-379
\bibitem{Mog91}{\sc ~A.~I. ~Mogilner}: Hamiltonians of solid state physics at
few-particle discrete Schr\"odinger operators: problems and results,
Advances in Sov. Math., {\bf 5} (1991), 139-194.

\bibitem{MumAli12}{\sc M.E. Muminov, N.M. Aliev:} Spectrum of the three-particle Schr\"odinger operator on a
one-dimensional lattice, Theoretical and Mathematical Physics,
{\bf171} (2012), pp 754-768.
\bibitem{Pankov A} {\sc A.~A. Pankov :} Lecture notice operators on Hilbert space. New York: Nova Science Publishers, 2007.

\bibitem{RSIII} {\sc M.Reed and B.Simon:}  Methods of modern mathematical physics. III:
Scattering theory,
 Academic Press, New York, 1979.

\bibitem{RSIV} {\sc M.Reed and
B.Simon:} {\it Methods of modern mathematical physics. IV: Analysis
of Operators,}
 Academic Press, New York, 1979.

\bibitem{Sob}  {\sc A.V. Sobolev}: The Efimov effect. Discrete spectrum
asymptotics, Commun. Math. Phys. {\bf 156} (1993), 127--168. S.

\bibitem{Tam}  {\sc H. Tamura}: Asymptotics for the number of negative eigenvalues of three-body
Schr\"{o}dinger operators with Efimov effect. Spectral and
scattering theory and applications, Adv. Stud. Pure Math., {\bf
23}(1994), Math. Soc. Japan, Tokyo, 311--322.

\bibitem{VZh83} {\sc S. A. Vugal'ter, G. M. Zhislin:} On finiteness of the discrete spectrum
of the energy operators of multiatomic molecules, Theoret. and Math.
Phys., 55:1 (1983), 357–365

\bibitem{Yaf}  {\sc D. R. Yafaev :} On the theory of the discrete spectrum of
the three-particle Schr\"{o}dinger operator, Math. USSR--Sb. {\bf
23} (1974), 535--559.
\bibitem{Letters}  {\sc K. Winkler,} {\sc G. Thalhammer,} {\sc F. Lang,} {\sc R. Frimm}
{\sc J. H. Denschlag,} {\sc A. J. Daley,} {\sc A. Kantian,} {\sc H.
P. B\"{u}chler} and {\sc P. Zoller:} Repulsively bound atom pairs in
an optical lattice. LETTERS. Vol. 441$|$15 June 2006$|$ {
doi:10.1038/nature04918} nature

\end{thebibliography}
\end{document}